\documentclass{amsart}
\usepackage[english]{babel}
\usepackage{amsmath,amsthm}
\usepackage[a4paper]{geometry}
\theoremstyle{plain}
\newtheorem{theorem}{Theorem}[section]

\newtheorem{lemma}[theorem]{Lemma}

\numberwithin{theorem}{section}
\numberwithin{equation}{section}
\numberwithin{table}{section}
\numberwithin{figure}{section}

\theoremstyle{definition}

\usepackage{fancyhdr}
\usepackage{enumerate}
\fancypagestyle{plain}{%
    \fancyhead{} 
}

\date{}
\begin{document}
\title{A note on the Neuman-S\'{a}ndor Mean}
\author{Tiehong Zhao}
\address{Department of mathematics, Hangzhou Normal University, Hangzhou 310036, P.R. China}%
\email{tiehongzhao@gmail.com}
\author{Yuming Chu}
\address{Department of mathematics, Huzhou Teachers College, Huzhou 313000, P.R. China}%
\email{chuyuming@hutc.zj.cn}
\author{Baoyu Liu}
\address{School of Science, Hangzhou Dianzi University, Hangzhou 310018, P.R. China}%
\email{627847649@qq.com} \subjclass[2000]{26E60}
\keywords{Neuman-S\'{a}ndor mean, harmonic mean, geometric mean,
quadratic mean, contra-harmonic mean.}
\thanks{This research was supported by the Natural Science Foundation of China under Grants 11071069 and 11171307,
and the Innovation Team Foundation of the Department of Education of Zhejiang Province under Grant T200924.}
\date{}
\maketitle

\begin{abstract}
In this article, we present the best possible upper and lower bounds
for the Neuman-S\'{a}ndor mean in terms of the geometric
combinations of harmonic and quadratic means, geometric and
quadratic means, harmonic and contra-harmonic means, and geometric
and contra-harmonic means.
\end{abstract}

\section{Introduction}

For $a,b>0$ with $a\neq b$ the Neuman-S\'{a}ndor mean $M(a,b)$ \cite{NS} is defined by
\begin{equation}\label{eqn:1.1}
M(a,b)=\frac{a-b}{2\sinh^{-1}\left(\frac{a-b}{a+b}\right)},
\end{equation}
where $\sinh^{-1}(x)=\log(x+\sqrt{x^2+1})$ is the inverse hyperbolic sine function.

Recently, the Neuman-S\'{a}ndor mean has been the subject of
intensive research. In particular, many remarkable inequalities for
the Neuman-S\'{a}ndor mean $M(a,b)$ can be found in the literature
\cite{NS,NS2}.

Let $H(a,b)=2ab/(a+b)$, $G(a,b)=\sqrt{ab}$, $L(a,b)=(b-a)/(\log
b-\log a)$, $P(a,b)=(a-b)/(4\arctan\sqrt{a/b}-\pi)$,
$A(a,b)=(a+b)/2$, $T(a,b)=(a-b)/[2\arctan((a-b)/(a+b))]$,
$Q(a,b)=\sqrt{(a^2+b^2)/2}$ and $C(a,b)=(a^2+b^2)/(a+b)$ be the
harmonic, geometric, logarithmic, first Seiffert, arithmetic, second
Seiffert, quadratic and contra-harmonic means of $a$ and $b$,
respectively. Then it is well-known that the inequalities
$$H(a,b)<G(a,b)<L(a,b)<P(a,b)<A(a,b)<M(a,b)<T(a,b)<Q(a,b)<C(a,b)$$
hold true for $a,b>0$ with $a\neq b$.

Neuman and S\'{a}ndor \cite{NS,NS2} established that
$$A(a,b)<M(a,b)<T(a,b),$$
$$P(a,b)M(a,b)<A^2(a,b),$$
$$A(a,b)T(a,b)<M^2(a,b)<[A^2(a,b)+T^2(a,b)]/2$$
hold for all $a,b>0$ with $a\neq b$.

Let $0<a,b<1/2$ with $a\neq b$, $a'=1-a$ and $b'=1-b$. Then the following Ky Fan inequalities
$$\frac{G(a,b)}{G(a',b')}<\frac{L(a,b)}{L(a',b')}<\frac{P(a,b)}{P(a',b')}<\frac{A(a,b)}{A(a',b')}<\frac{M(a,b)}{M(a',b')}<\frac{T(a,b)}{T(a',b')}$$
were presented in \cite{NS}.

Li et al. \cite{LLC} showed that the double inequality
$L_{p_0}(a,b)<M(a,b)<L_2(a,b)$ holds for all $a,b>0$ with
$a\neq b$, where
$L_p(a,b)=\left[(b^{p+1}-a^{p+1})/((p+1)(b-a))\right]^{1/p}
(p\neq-1,0)$, $L_0=1/e(b^b/a^a)^{1/(b-a)}$ and
$L_{-1}(a,b)=(b-a)/(\log b-\log a)$ be the $p-$th generalized
logarithmic mean of $a$ and $b$, and $p_0=1.843\cdots$ is the unique
solution of the equation $(p+1)^{1/p}=2\log(1+\sqrt{2})$.

In \cite{Ne}, Neuman proved that the double inequalities
$$Q^{\alpha}(a,b)A^{1-\alpha}(a,b)<M(a,b)<Q^{\beta}(a,b)A^{1-\beta}(a,b)$$
and
$$C^{\lambda}(a,b)A^{1-\lambda}(a,b)<M(a,b)<C^{\mu}(a,b)A^{1-\mu}(a,b)$$
hold for all $a,b>0$ with $a\neq b$ if and only if $\alpha\leq 1/3$,
$\beta\geq
2\left(\log(2+\sqrt{2})-\log{3}\right)/\log{2}=0.373\cdots$,
$\lambda\leq1/6$ and
$\mu\geq\left(\log(2+\sqrt{2})-\log{3}\right)/\log{2}=0.186\cdots$.

The main purpose of this paper is to find the least values
$\alpha_{1}$, $\alpha_{2}$, $\alpha_{3}$, $\alpha_{4}$, and the
greatest values $\beta_{1}$, $\beta_{2}$, $\beta_{3}$, $\beta_{4}$
such that the double inequalities
$$H^{\alpha_{1}}(a,b)Q^{1-\alpha_{1}}(a,b)<M(a,b)<H^{\beta_{1}}(a,b)Q^{1-\beta_{1}}(a,b),$$
$$G^{\alpha_{2}}(a,b)Q^{1-\alpha_{2}}(a,b)<M(a,b)<G^{\beta_{2}}(a,b)Q^{1-\beta_{2}}(a,b),$$
$$H^{\alpha_{3}}(a,b)C^{1-\alpha_{3}}(a,b)<M(a,b)<H^{\beta_{3}}(a,b)C^{1-\beta_{3}}(a,b)$$
and
$$G^{\alpha_{4}}(a,b)C^{1-\alpha_{4}}(a,b)<M(a,b)<G^{\beta_{4}}(a,b)C^{1-\beta_{4}}(a,b)$$
hold for all $a,b>0$ with $a\neq b$.

\section{Lemmas}

In order to establish our main results we need four lemmas, which we
present in this section.

\medskip
\begin{lemma}\label{le:2.1}
(See \cite{AVV}, Theorem 1.25). For $-\infty<a<b<\infty$, let
$f,g:[a,b]\rightarrow{\mathbb{R}}$ be continuous on $[a,b]$, and be
differentiable on $(a,b)$, let $g'(x)\neq 0$ on $(a,b)$. If
$f^{\prime}(x)/g^{\prime}(x)$ is increasing (decreasing) on $(a,b)$,
then so are
$$\frac{f(x)-f(a)}{g(x)-g(a)}\ \ \mbox{and}\ \ \frac{f(x)-f(b)}{g(x)-g(b)}.$$
If $f^{\prime}(x)/g^{\prime}(x)$ is strictly monotone, then the
monotonicity in the conclusion is also strict.
\end{lemma}

\medskip
\begin{lemma}\label{le:2.2}
(See \cite{SV}, Lemma 1.1). Suppose that the power series
$f(x)=\sum\limits_{n=0}^{\infty}a_{n}x^{n}$ and
$g(x)=\sum\limits_{n=0}^{\infty}b_{n}x^{n}$ have the radius of
convergence $r>0$ and $b_{n}>0$ for all $n\in\{0,1,2,\cdots\}$. Let
$h(x)={f(x)}/{g(x)}$, then

(1) If the sequence $\{a_{n}/b_{n}\}_{n=0}^{\infty}$ is (strictly)
increasing (decreasing), then $h(x)$ is also (strictly) increasing
(decreasing) on $(0,r)$;

(2) If the sequence $\{a_{n}/b_{n}\}$ is (strictly) increasing
(decreasing) for $0<n\leq n_{0}$ and (strictly) decreasing
(increasing) for $n>n_{0}$, then there exists $x_{0}\in(0,r)$ such
that $h(x)$ is (strictly) increasing (decreasing) on $(0,x_{0})$ and
(strictly) decreasing (increasing) on $(x_{0},r)$.
\end{lemma}

\medskip
\begin{lemma}\label{le:2.3}
Let
\begin{equation}\label{eqn:2.1}
\phi(t)=\frac{[3-\cosh(2t)][\sinh(2t)-2t]}{2t\sinh^{2}(t)[5+\cosh(2t)]},
\end{equation}
then $\phi(t)$ is strictly decreasing in $(0,\log(1+\sqrt{2}))$,
where $\sinh(t)=(e^t-e^{-t})/2$ and $\cosh(t)=(e^t+e^{-t})/2$ are
respectively the hyperbolic sine and cosine functions.
\end{lemma}

\begin{proof}
Let us denote by  $\phi_{1}(t)$ and  $\phi_{2}(t)$ respectively the
numerator  and denominator of (2.1) expand the factor to obtain
\begin{equation}\label{eqn:2.2}
\phi_{1}(t)=3\sinh(2t)-6t+2t\cosh(2t)-\frac{1}{2}\sinh(4t),
\end{equation}
\begin{equation}\label{eqn:2.3}
\phi_{2}(t)=\frac{t}{2}\left[8\cosh(2t)+\cosh(4t)-9\right].
\end{equation}

Using the power series
$\sinh(t)=\sum_{n=0}^{\infty}t^{2n+1}/(2n+1)!$ and
$\cosh(t)=\sum_{n=0}^{\infty}t^{2n}/(2n)!$, we can express
(\ref{eqn:2.2}) and (\ref{eqn:2.3}) as follows
\begin{equation}\label{eqn:2.4}
\phi_{1}(t)=\sum_{n=1}^{\infty}\frac{2^{2n+1}(2n+4-2^{2n})}{(2n+1)!}t^{2n+1}=t^3\sum_{n=0}^{\infty}\frac{2^{2n+4}(n+3-2^{2n+1})}{(2n+3)!}t^{2n},
\end{equation}
\begin{equation}\label{eqn:2.5}
\phi_{2}(t)=\sum_{n=1}^{\infty}\frac{2^{2n}(4+2^{2n-1})}{(2n)!}t^{2n+1}=t^3\sum_{n=0}^{\infty}\frac{2^{2n+4}(1+2^{2n-1})}{(2n+2)!}t^{2n}.
\end{equation}

It follows from (\ref{eqn:2.4}) and (\ref{eqn:2.5}) that
\begin{equation}\label{eqn:2.6}
\phi(t)=\frac{\sum\limits_{n=0}^{\infty}a_{n}t^{2n}}{\sum\limits_{n=0}^{\infty}b_{n}t^{2n}}
\end{equation}
with $a_{n}=2^{2n+4}(n+3-2^{2n+1})/(2n+3)!$ and
$b_{n}=2^{2n+4}(1+2^{2n-1})/(2n+2)!$.

Let $c_{n}=a_{n}/b_{n}$, then simple computations lead to
\begin{equation*}
c_{n}=\frac{(n+3)-2^{2n+1}}{(2n+3)(1+2^{2n-1})},
\end{equation*}
\begin{equation}\label{eqn:2.7}
c_{0}=\frac{2}{9}>c_{1}=-\frac{4}{15}>c_{2}=-\frac{3}{7}<c_{3}=-\frac{122}{297},
\end{equation}
\begin{align}\label{eqn:2.8}
c_{n+1}-c_{n}=&\frac{2^{4n+3}-(6n^2+57n+76)2^{2n-1}-3}{(2n+3)(2n+5)(1+2^{2n-1})(1+2^{2n+1})}\nonumber\\
=&\frac{[2(4^{n}-38)+6(4^n-n^2)+(128\times
4^{n-2}-57n)]2^{2n-1}-3}{(2n+3)(2n+5)(1+2^{2n-1})(1+2^{2n+1})}>0
\end{align}
for all $n>2$.

Inequalities (\ref{eqn:2.7}) and (\ref{eqn:2.8}) implies that the
sequence $\{a_{n}/b_{n}\}$ is strictly decreasing in $0<n\leq 2$ and
strictly increasing for $n>2$, then from  (\ref{eqn:2.6}) and Lemma
2.2(2) we know that there exists $t_{0}>0$ such that $\phi(t)$ is
strictly decreasing on $(0,t_{0})$ and strictly increasing in
$(t_{0},\infty)$.

For convenience, let us denote
$t^{\star}=\log(1+\sqrt{2})=0.881\cdots$, then we have
\begin{equation}\label{eqn:2.9}
\sinh(t^*)=1,\quad \sinh(2t^*)=2\sqrt{2},\quad \sinh(3t^*)=7,
\end{equation}
\begin{equation}\label{eqn:2.10}
\cosh(t^*)=\sqrt{2},\quad \cosh(2t^*)=3,\quad \cosh(3t^*)=5\sqrt{2}.
\end{equation}
Differentiating (\ref{eqn:2.1}) yields
\begin{equation}\label{eqn:2.11}
{\phi}'(t)=\frac{{\phi_{1}}'(t){\phi_{2}}(t)-{\phi_{1}}(t){\phi_{2}}'(t)}{{\phi_{2}}^2(t)},
\end{equation}
where
\begin{equation}\label{eqn:2.12}
\phi_{1}'(t)=8\sinh(t)[t\cosh(t)-2\sinh^{3}(t)],
\end{equation}
\begin{equation}\label{eqn:2.13}
\phi_{2}'(t)=\sinh(t)[20t\cosh(t)+4t\cosh(3t)+9\sinh(t)+\sinh(3t)].
\end{equation}

From (\ref{eqn:2.2}) and (\ref{eqn:2.3}) together with
(\ref{eqn:2.9})-(\ref{eqn:2.13}) we get
\begin{equation}\label{eqn:2.14}
{\phi}'(t^*)=-\frac{\sqrt{2}-t^*}{\sqrt{2}t^*}<0.
\end{equation}

It follows from the piecewise monotonicity of $\phi(t)$ and
(\ref{eqn:2.14}) that $t_{0}>t^*$. This completes the proof of Lemma
2.3.
\end{proof}

\medskip
\begin{lemma}\label{le:2.4}
Let $p\in[0,1)$, and
\begin{equation}\label{eqn:2.15}
\varphi_{p}(t)=\log(1+x^2)-\log\frac{x}{\sinh^{-1}(x)}+p\left[\frac{1}{2}\log(1-x^2)-\log(1+x^2)\right].
\end{equation}
Then $\varphi_{5/9}(x)<0$ and $\varphi_{0}(x)>0$ for all
$x\in(0,1)$.
\end{lemma}

\begin{proof}
From (\ref{eqn:2.15}) one has
\begin{equation}\label{eqn:2.16}
\varphi_{p}(0^+)=0,
\end{equation}
\begin{equation}\label{eqn:2.17}
\varphi_{p}'(x)=\frac{\phi_{p}(x)}{x(1-x^4)\sqrt{1+x^2}\sinh^{-1}(x)},
\end{equation}
where
\begin{equation}\label{eqn:2.18}
\phi_{p}(x)=x-x^5-[1+(3p-2)x^2+(1-p)x^4]\sqrt{1+x^2}\sinh^{-1}(x).
\end{equation}

We divide the proof into two cases.

{\bf Case 1} $p=5/9$. Then (\ref{eqn:2.18}) leads to
\begin{equation}\label{eqn:2.19}
\phi_{5/9}(0)=0,
\end{equation}
\begin{equation}\label{eqn:2.20}
\phi_{5/9}'(x)=-\frac{xf(x)}{9\sqrt{1+x^2}},
\end{equation}
where
\begin{equation}\label{eqn:2.21}
f(x)=x(49x^2-3)\sqrt{1+x^2}+(3+7x^2+20x^4)\sinh^{-1}(x),
\end{equation}
\begin{equation}\label{eqn:2.22}
f(0)=0.
\end{equation}
Differentiating (\ref{eqn:2.21}) yields
\begin{equation}\label{eqn:2.23}
f'(x)=\frac{2x[74x+108x^3+(7+40x^2)\sqrt{1+x^2}\sinh^{-1}(x)]}{\sqrt{1+x^2}}>0
\end{equation}
for $x\in(0,1)$.

Therefore, $\phi_{5/9}(x)<0$ for all $x\in(0,1)$ follows easily from
(\ref{eqn:2.19}) and (\ref{eqn:2.20}) together with (\ref{eqn:2.22})
and (\ref{eqn:2.23}).

\medskip

{\bf Case 2} $p=0$. Then (\ref{eqn:2.18}) yields
\begin{equation}\label{eqn:2.24}
\frac{\phi_{0}(x)}{1-x^2}=x(1+x^2)-(1-x^2)\sqrt{1+x^2}\sinh^{-1}(x):=g(x),
\end{equation}
\begin{equation}\label{eqn:2.25}
g(0)=0.
\end{equation}
Differentiating (\ref{eqn:2.24}) we get
\begin{equation}\label{eqn:2.26}
g'(x)=\frac{x[4x\sqrt{1+x^2}+(1+3x^2)\sinh^{-1}(x)]}{\sqrt{1+x^2}}>0
\end{equation}
for $x\in(0,1)$

Therefore, $\varphi_{0}(x)>0$ for $x\in(0,1)$ easily from
(\ref{eqn:2.16}) and (\ref{eqn:2.17}) together with
(\ref{eqn:2.24})-(\ref{eqn:2.26}).

\medskip

\end{proof}

\section{Bounds for the Neuman-S\'{a}ndor Mean}

In this section we will deal with problems of finding sharp bounds
for the Neuman-S\'{a}ndor Mean $M(a,b)$ in terms of the geometric
combinations of harmonic mean $H(a,b)$ and quadratic mean $Q(a,b)$,
geometric mean $G(a,b)$ and quadratic mean $Q(a,b)$, harmonic mean
$H(a,b)$ and contra-harmonic mean $C(a,b)$, and geometric mean
$G(a,b)$ and contra-harmonic mean $C(a,b)$.

Since $H(a,b)$, $G(a,b)$, $M(a,b)$, $Q(a,b)$ and $C(a,b)$ are
symmetric and homogeneous of degree $1$. Without loss of generality,
we assume that $a>b$. For the later use we denote
$x=(a-b)/(a+b)\in(0,1)$ and $t=\sinh^{-1}(x)\in(0,t^*)$ with
$t^*=\log(1+\sqrt{2})=0.881\cdots$.

\bigskip
\begin{theorem}\label{th:3.1}The double inequality
\begin{equation}\label{eqn:3.1}
H^{\alpha}(a,b)Q^{1-\alpha}(a,b)<M(a,b)<H^{\beta}(a,b)Q^{1-\beta}(a,b)
\end{equation}
holds true for all $a,b>0$ with $a\neq b$ if and only if $\alpha\geq
2/9$ and $\beta\leq 0$.
\end{theorem}

\begin{proof} First we take the logarithm of each member of
(\ref{eqn:3.1}) and next rearrange terms to obtain
\begin{equation}\label{eqn:3.2}
\beta<\frac{\log[Q(a,b)]-\log[M(a,b)]}{\log[Q(a,b)]-\log[H(a,b)]}<\alpha.
\end{equation}
Note that
\begin{equation}\label{eqn:3.3}
\frac{M(a,b)}{A(a,b)}=\frac{x}{\sinh^{-1}(x)},\quad
\frac{H(a,b)}{A(a,b)}=1-x^2,\quad
\frac{Q(a,b)}{A(a,b)}=\sqrt{1+x^2}.
\end{equation}
Use of (\ref{eqn:3.3}) followed by a substitution
$x=\sinh(t)(0<t<t^*)$, inequality (\ref{eqn:3.2}) becomes
\begin{equation}\label{eqn:3.4}
\beta<f(t)<\alpha,
\end{equation}
where
\begin{equation}\label{eqn:3.5}
f(t)=\frac{\log[\cosh(t)]-\log[\sinh(t)/t]}{\log[\cosh(t)]-\log[1-\sinh^2(t)]}:=\frac{f_{1}(t)}{f_{2}(t)}.
\end{equation}

In order to use Lemma 2.1, we consider the following
\begin{equation}\label{eqn:3.6}
\frac{f_{1}'(t)}{f_{2}'(t)}=\frac{[3-\cosh(2t)][\sinh(2t)-2t]}{2t\sinh^{2}(t)[5+\cosh(2t)]}:=\phi(t),
\end{equation}
where $\phi(t)$ is defined as in Lemma 2.3.

It follows from Lemmas 2.1 and 2.3 together with (\ref{eqn:3.6})
that
\begin{equation*}
f(t)=\frac{f_{1}(t)}{f_{2}(t)}=\frac{f_{1}(t)-f_{1}(0^+)}{f_{2}(t)-f_{2}(0)}
\end{equation*}
is strictly decreasing on $(0,t^*)$. This in turn implies that
\begin{equation}\label{eqn:3.7}
\lim_{t\rightarrow 0^+}f(t)=\frac{2}{9},\quad \lim_{t\rightarrow
t^*}f(t)=0.
\end{equation}

Making use of (\ref{eqn:3.7}) and the monotonicity of $\phi(t)$ we
conclude that in order for the double inequality (\ref{eqn:3.1}) to
be valid it is necessary and sufficient that $\alpha\geq 2/9$ and
$\beta\leq 0$.
\end{proof}

\bigskip
\begin{theorem}\label{th:3.2}The two-sided inequality
\begin{equation}\label{eqn:3.8}
G^{\alpha}(a,b)Q^{1-\alpha}(a,b)<M(a,b)<G^{\beta}(a,b)Q^{1-\beta}(a,b)
\end{equation}
holds true for all $a,b>0$ with $a\neq b$ if and only if $\alpha\geq
1/3$ and $\beta\leq 0$.
\end{theorem}

\begin{proof} We will follows lines introduced in the proof of Theorem 3.1.
We take the logarithm of each member of (\ref{eqn:3.8}) and next
rearrange terms to get
\begin{equation}\label{eqn:3.9}
\beta<\frac{\log[Q(a,b)]-\log[M(a,b)]}{\log[Q(a,b)]-\log[G(a,b)]}<\alpha.
\end{equation}
Use of (\ref{eqn:3.3}) and $G(a,b)/A(a,b)=\sqrt{1-x^2}$ followed by
a substitution $x=\sinh(t)(0<t<t^*)$, inequality (\ref{eqn:3.9}) is
equivalent to
\begin{equation}\label{eqn:3.10}
\beta<g(t)<\alpha,
\end{equation}
where
\begin{equation}\label{eqn:3.11}
g(t)=\frac{\log[\cosh(t)]-\log[\sinh(t)/t]}{\log[\cosh(t)]-\log[1-\sinh^2(t)]/2}:=\frac{g_{1}(t)}{g_{2}(t)}.
\end{equation}
Equation (\ref{eqn:3.11}) leads to
\begin{align}\label{eqn:3.12}
\frac{g_{1}'(t)}{g_{2}'(t)}=&\frac{[3-\cosh(2t)][\sinh(2t)-2t]}{8t\sinh^{2}(t)}
=\frac{\sum\limits_{n=1}^{\infty}[2^{2n+1}(2n+4-2^{2n})/(2n+1)!]t^{2n+1}}{\sum\limits_{n=1}^{\infty}[2^{2n+2}/(2n)!]t^{2n+1}}\nonumber\\
=&\frac{\sum\limits_{n=0}^{\infty}[2^{2n+4}(n+3-2^{2n+1})/(2n+3)!]t^{2n}}{\sum\limits_{n=0}^{\infty}[2^{2n+4}/(2n+2)!]t^{2n}}
:=\frac{\sum\limits_{n=0}^{\infty}a_{n}'t^{2n}}{\sum\limits_{n=0}^{\infty}b_{n}'t^{2n}},
\end{align}
\begin{equation}\label{eqn:3.13}
\frac{a_{n+1}'}{b_{n+1}'}-\frac{a_{n}'}{b_{n}'}=-\frac{3+(6n+7)2^{2n+1}}{(2n+3)(2n+5)}<0
\end{equation}
for all $n\in\{0,1,2,\cdots\}$.

It follows from Lemmas 2.1(1) and  (\ref{eqn:3.12}) together with
(\ref{eqn:3.13}) that $g_{1}'(t)/g_{2}'(t)$ is strictly decreasing
on $(0,t^*)$.

From Lemma 2.1 and (\ref{eqn:3.11}) together with
$g_{1}(0^+)=g_{2}(0)=0$ and the monotonicity of
$g_{1}'(t)/g_{2}'(t)$ we clearly see that $g(t)$ is strictly
decreasing on $(0,t^*)$.

Therefore, Theorem 3.2 follows from the monotonicity of $g(t)$ and
(\ref{eqn:3.10}) together with the fact that
\begin{equation*}
\lim_{t\rightarrow 0^+}g(t)=\frac{1}{3},\quad \lim_{t\rightarrow
t^*}g(t)=0.
\end{equation*}
\end{proof}

\begin{theorem}\label{th:3.2}The following simultaneous inequality
\begin{equation}\label{eqn:3.14}
H^{\alpha}(a,b)C^{1-\alpha}(a,b)<M(a,b)<H^{\beta}(a,b)C^{1-\beta}(a,b)
\end{equation}
holds true for all $a,b>0$ with $a\neq b$ if and only if $\alpha\geq
5/12$ and $\beta\leq 0$.
\end{theorem}

\begin{proof} We take the logarithm of each member of (\ref{eqn:3.14}) and next
rearrange terms to get
\begin{equation}\label{eqn:3.15}
\beta<\frac{\log[C(a,b)]-\log[M(a,b)]}{\log[C(a,b)]-\log[H(a,b)]}<\alpha.
\end{equation}
Use of (\ref{eqn:3.3}) and $C(a,b)/A(a,b)=1+x^2$ followed by a
substitution $x=\sinh(t)(0<t<t^*)$, inequality (\ref{eqn:3.15})
becomes
\begin{equation}\label{eqn:3.16}
\beta<h(t)<\alpha,
\end{equation}
where
\begin{equation}\label{eqn:3.17}
h(t)=\frac{\log[\cosh(t)]-\log[\sinh(t)/t]/2}{\log[\cosh(t)]-\log[1-\sinh^2(t)]/2}:=\frac{h_{1}(t)}{h_{2}(t)}.
\end{equation}
Equation (\ref{eqn:3.17}) gives
\begin{align}\label{eqn:3.18}
\frac{h_{1}'(t)}{h_{2}'(t)}=&\frac{[3-\cosh(2t)][\sinh(2t)+t\cosh(2t)-3t]}{16t\sinh^{2}(t)}\nonumber\\
=&\frac{\sum\limits_{n=0}^{\infty}\left[2^{2n+3}\left((3-2^{2n})(2n+3)+3-2^{2n+2}\right)/(2n+3)!\right]t^{2n}}
{\sum\limits_{n=0}^{\infty}[2^{2n+5}/(2n+2)!]t^{2n}}
:=\frac{\sum\limits_{n=0}^{\infty}c_{n}'t^{2n}}{\sum\limits_{n=0}^{\infty}d_{n}'t^{2n}},
\end{align}
\begin{equation}\label{eqn:3.19}
\frac{c_{n+1}'}{d_{n+1}'}-\frac{c_{n}'}{d_{n}'}=-3\times
2^{2n-2}-\frac{3}{2(2n+3)(2n+5)}-\frac{(6n+7)2^{2n}}{(2n+3)(2n+5)}<0
\end{equation}
for all $n\in\{0,1,2,\cdots\}$.

It follows from Lemmas 2.2(1) and  (\ref{eqn:3.18}) together with
(\ref{eqn:3.19}) that $h_{1}'(t)/h_{2}'(t)$ is strictly decreasing
on $(0,t^*)$.

From Lemma 2.1 and (\ref{eqn:3.17}) together with
$h_{1}(0^+)=h_{2}(0)=0$ and the monotonicity of
$h_{1}'(t)/h_{2}'(t)$ we clearly see that $h(t)$ is strictly
decreasing on $(0,t^*)$.

Therefore, Theorem 3.3 follows from the monotonicity of $h(t)$ and
(\ref{eqn:3.16}) together with the fact that
\begin{equation*}
\lim_{t\rightarrow 0^+}h(t)=\frac{5}{12},\quad \lim_{t\rightarrow
t^*}h(t)=0.
\end{equation*}
\end{proof}

\begin{theorem}\label{th:3.4}The following inequality
\begin{equation}\label{eqn:3.20}
G^{\alpha}(a,b)C^{1-\alpha}(a,b)<M(a,b)<G^{\beta}(a,b)C^{1-\beta}(a,b)
\end{equation}
is valid for all $a,b>0$ with $a\neq b$ if and only if $\alpha\geq
5/9$ and $\beta\leq 0$.
\end{theorem}

\begin{proof} Making use of (\ref{eqn:3.3}) and
$C(a,b)/A(a,b)=1+x^2$ together with  $G(a,b)/A(a,b)=\sqrt{1-x^2}$ we
get
\begin{equation}\label{eqn:3.21}
\frac{\log[C(a,b)]-\log[M(a,b)]}{\log[C(a,b)]-\log[G(a,b)]}
=\frac{\log(1+x^2)-\log[x/\sinh^{-1}(x)]}{\log(1+x^2)-\log\sqrt{1-x^2}}.
\end{equation}

Elaborated computations lead to
\begin{equation}\label{eqn:3.22}
\lim_{x\rightarrow
0^+}\frac{\log(1+x^2)-\log[x/\sinh^{-1}(x)]}{\log(1+x^2)-\log\sqrt{1-x^2}}=\frac{5}{9},
\end{equation}
\begin{equation}\label{eqn:3.23}
\lim_{x\rightarrow
1^-}\frac{\log(1+x^2)-\log[x/\sinh^{-1}(x)]}{\log(1+x^2)-\log\sqrt{1-x^2}}=0.
\end{equation}

Taking the logarithm of (\ref{eqn:3.20}), we consider the difference
between the convex combination of $\log{G(a,b)}$, $\log{C(a,b)}$ and
$\log{M(a,b)}$ as follows
\begin{align}\label{eqn:3.24}
&p\log{G(a,b)}+(1-p)\log{C(a,b)}-\log{M(a,b)}\nonumber\\
=&p\log\sqrt{1-x^2}+(1-p)\log(1+x^2)-\log\frac{x}{\sinh^{-1}(x)}=\varphi_{p}(x),
\end{align}
where $\varphi_{p}(x)$ is defined as in Lemma 2.4.

Therefore, $G^{5/9}(a,b)C^{4/9}(a,b)<M(a,b)<C(a,b)$ for all $a,b>0$
with $a\neq b$ follows from (\ref{eqn:3.24}) and Lemma 2.4. This in
conjunction with the following statements gives the asserted result.
\setlength\leftmargini{.3cm}
\begin{itemize}
\item If $\alpha<5/9$, then equations (\ref{eqn:3.21}) and (\ref{eqn:3.22})
lead to the conclusion that there exists $0<\delta_1<1$ such that
$M(a,b)<G^{\alpha}(a,b)C^{1-\alpha}(a,b)$ for all $a,b>0$ with
$(a-b)/(a+b)\in(0,\delta_1)$.

\item If $\beta>0$, then equations  (\ref{eqn:3.21}) and (\ref{eqn:3.23}) imply that
there exists $0<\delta_2<1$ such that
$M(a,b)>G^{\beta}(a,b)C^{1-\beta}(a,b)$ for all $a,b>0$ with
$(a-b)/(a+b)\in(1-\delta_2,1)$.
\end{itemize}
\end{proof}

\end{document}